\documentclass[12pt]{amsart}
\usepackage{graphicx}
\usepackage{amssymb}
\usepackage{amsmath}
\usepackage{amsthm,amsfonts,bbm}
\usepackage{amscd}
\usepackage{geometry}
\usepackage{epsfig,epstopdf}
\usepackage{mathrsfs}
\usepackage{cite}
\usepackage[backref,colorlinks=true,anchorcolor=blue,filecolor=blue,linkcolor=blue,urlcolor=blue,citecolor=blue]{hyperref}
\usepackage{color}

\theoremstyle{plain}
\newtheorem{thm}{Theorem}[section]
\newtheorem{coro}[thm]{Corollary}
\newtheorem{lemma}[thm]{Lemma}
\theoremstyle{definition}

\theoremstyle{plain}
\newtheorem{rmk}[thm]{Remark}
\theoremstyle{definition}
\newtheorem{claim}{Claim}

\numberwithin{equation}{section}
\numberwithin{figure}{section}
\def \D{\Delta}
\def \d{\delta}
\def \dn{\text{d}}
\def \deg{\text{deg}}
\def \down{\text{down}}
\def \ex{\text{ex}}
\def \im{\text{im}}
\def \ker{\text{ker}}
\def \K{\mathcal{K}}

\def \la{\lambda}
\def \p{\partial}
\def \q{\mathfrak{q}}
\def \s{\sigma}
\def \supp{\text{supp}}
\def \R{\mathbb{R}}
\def \T{\wedge}
\def \un{\text{u}}
\def \up{\text{up}}

\usepackage{amsaddr}

\DeclareMathOperator{\sgn}{sgn}

\begin{document}
\title[Spectral radius of complexes]{Signless Laplacian spectral radius of simplicial complexes without holes}

\author[Y.-Z. Fan, C.-M. She, H.-Z. Zhang]{Yi-Zheng Fan*, Chuan-Ming She, Huan-Zhi Zhang}
\address{\small Center for Pure Mathematics, School of Mathematical Sciences, \\ Anhui University, Hefei 230601, P. R. China}
\email{fanyz@ahu.edu.cn, shecm@stu.ahu.edu.cn, zhanghz@stu.ahu.edu.cn}
\thanks{*The corresponding author.
Supported by National Natural Science Foundation of China (No. 12331012).
}



\begin{abstract}
We study a spectral analog of the Tur\'an problem for simplicial complexes. Specifically, we consider the extremal problem of maximizing the signless Laplacian spectral radius among simplicial complexes without holes. We determine the structure of the simplicial complex attaining the maximum spectral radius, extending classical extremal results for graphs without cycles to the setting of higher-dimensional simplicial complexes.
More generally, we establish an upper bound on the signless Laplacian spectral radius of simplicial complexes with prescribed Betti numbers. As an application, using the connection between the signless Laplacian spectral radius and the face numbers of a simplicial complex, we derive bounds on Turán numbers for both hypergraphs and simplicial complexes.
Our technique involves the canonical Alexander dual of perfect matchings and coloring of simplicial complexes.
\end{abstract}

\maketitle

\section{Introduction}
In extremal combinations, a central problem  is  to determine the maximum number of edges in a graph on $n$ vertices that does not contain a given graph as a subgraph, which is known as the Tur\'an problem.
The problem naturally extends to hypergraphs.
Formally, let $F$ be an $r$-uniform hypergraph (or simply an \emph{$r$-graph}).
A $r$-graph $H$ is said to be \emph{$F$-free} if it contains no sub-hypergraphs isomorphic to $F$.
The \emph{Tur\'an number} $\ex_r(n,F)$ is the maximum number of edges in an $F$-free $r$-graph on $n$ vertices.
In the graph case (i.e. $r=2$), the celebrated Erd\H{o}s-Stone-Simonovits theorem \cite{ERD1966} shows that
$$ \ex_2(n,F)=(1+o(1))\left(1-\frac{1}{\chi(F)-1}\right)\binom{n}{2},$$
where $\chi(F)$ denotes the chromatic number of $F$.
The problem becomes difficult in the hypergraph case when $r \ge 3$.
Let $\D_n^r$ denote the complete $r$-graph on $n$ vertices.
Even for the first nontrivial case of $r=3$ and $n=4$,
the exact value of the Tur\'an number $\ex_r(n,\D_{4}^{3})$ remains unknown.
Tur\'an conjectured the following exact formula \cite{Tur1941}:
\begin{equation}\label{Turan}
\ex(n,\D_{4}^{3})=
\begin{cases}\frac{t^{2}(5t-3)}{2}, & \text{if~} n=3t; \\
\frac{t(5t^2+2t-1)}{2}, & \text{if~} n=3t+1; \\
\frac{t(t+1)(5t+2)}{2}, & \text{if~} n=3t+2.
\end{cases}
\end{equation}
The \emph{Tur\'an density} of $F$, denoted by $\pi(F)$ \cite{KNS1964},  is defined as
$$\pi(F)=\lim_{n \to \infty}\frac{\ex_r(n,F)}{\binom{n}{r}}.$$
The above conjecture, if true, would give $\pi(\D_{4}^{3})=\frac{5}{9}$.
The best known upper bound is $\pi(\D_{4}^{3})\le\frac{3+\sqrt{17}}{12}$, due to Chung and Lu \cite{CL1999}.
A general upper bound
\begin{equation}\label{DeCaen}
\pi(\D_{r+1}^{r})\le\frac{r-1}{r},
\end{equation}
was established by de Caen \cite{DE1983}, and subsequently improved for odd values of $r$ in \cite{CL1999} and for even values in \cite{LZ2009}.
Beyond this, various Turán-type problems have been explored, such as determining the Turán density of $\D_{t+1}^r$ \cite{M2006}, the Tur\'an number of expansion of the Fan graph \cite{MP2007}, and the Berge hypergraph \cite{GP2017}.
In particular, the Tur\'an numbers of linear hypergraphs have attracted significant attention; see \cite{RS1978,EFR1986,LV2003,Tim2017,FO2017,GMV2019,EGM2019,FG2020,GC2021}.

The spectral version of the hypergraph Tur\'an problem has attracted a great deal of attention in recent years through the use of adjacency tensors.
Keevash, Lenz, and Mubayi~\cite{KLM2014} established two general criteria under which spectral extremal results may be deduced from `strong stability' forms of the corresponding extremal results.
A growing line of work focuses on determining the maximum spectral radius of linear $r$-uniform hypergraphs that forbid certain subhypergraphs; see  \cite{HCC2021,GCH2022,SFKH2023,SFK2025}.

In this paper we view the hypergraphs from the perspective of simplicial complexes (or simply called complexes).
An $r$-graph can be regarded as a pure (abstract) simplicial complex of dimension $r-1$,
where all subsets of each edge are regarded as faces of the simplicial complex.
Under this correspondence, the complete $r$-graph $\D_{r+1}^r$ corresponds to a pure $(r-1)$-dimensional complex on $r+1$ vertices, whose facets are all possible $r$-subsets.
Consequently, the Tur\'an number $\ex_r(n,\D_{r+1}^{r})$ is exactly the maximum number of facets in a pure $(r-1)$-dimensional complex that does not contain a subcomplex isomorphic to  $\D_{r+1}^r$.
In general, for $\mathcal{H}$, a family of $r$-dimensional complexes, denote by $\ex(n,\mathcal{H})$ the maximum number of $r$-faces in a pure $r$-dimensional complex on $n$ vertices that does not
contain any complex in $\mathcal{H}$ as a subcomplex.
Usually, $\mathcal{H}$ is some collection of subcomplexes that are all homeomorphic to a
particular fixed topological space, e.g., a sphere $S^r$ of dimension $r$.

Brown, Erd\H{o}s, and S\'os \cite{SEB1973} proved that
\begin{equation}\label{nS2}
\ex(n,S^2)=\Theta(n^{\frac{5}{2}}).
\end{equation}
More recently, Newman and Pavelka \cite{NP2024} provided a conditional lower bound for $\ex(n,S^r)$ and showed that
\begin{equation}\label{nSr}
\ex(n,S^r)\le O\left(n^{r+1-1/2^{r-1})}\right).
\end{equation}
In general, homeomorph Tur\'an problems for arbitrary complexes have been investigated by Keevash, Long, Narayanan, and Scott \cite{KLNS2021}, and by Long, Narayanan, and Yap~\cite{LNY2022}.
The asymptotics for homeomorphs of fixed orientable surfaces in $2$-dimensional complexes were established by Kupavskii, Polyanskii, Tomon, and Zakharov \cite{KPTZ2022},
and those for homeomorphs of fixed non-orientable surfaces were established by Sankar \cite{San}.

In this paper, we investigate the spectral version of the Tur\'an problem of simplicial complexes using the signless Laplacian operator.
Let $K$ be a pure $r$-dimensional complex, $Q_{r-1}^{\up}(K)$ be the $(r-1)$-th up signless Laplacian operator of $K$, and $\q_{r-1}(K)$ be the spectral radius of $Q_{r-1}^{\up}(K)$.
We first prove a key lemma for characterizing the complex with the maximum number of facets or the maximum spectral radius.

\begin{lemma}\label{pathcon}
Let $K$ be a pure $r$-dimensional complex on $n$ vertices.
There exists a pure $r$-dimensional complex $L$ on $n$ vertices with the following properties:

{\em(1)} $L$ contains $K$ as a subcomplex;

{\em(2)} $L$ contains all possible $(r-1)$-faces;

{\em(3)} $L$ is $(r-1)$-path connected;

{\em(4)} $L$ contains no new $r$-holes except those of $K$.

\end{lemma}

Let $\K(n,r,\beta)$ denote the set of pure $r$-dimensional complexes on $n$ vertices with the $r$-th Betti number $\beta$.
By Lemma \ref{pathcon}, we immediately obtain the following corollary on the maximum number of facets.

\begin{coro}\label{e-con}
Let $K$ be a pure $r$-dimensional complex on $n$ vertices with the maximum number of facets among all complexes in $\mathcal{K}(n,r,\beta)$.
Then $K$ contains all possible $(r-1)$-faces and  is $(r-1)$-path connected.
\end{coro}

Corollary \ref{e-con} is a generalization of the graph case, namely, the $1$-dimensional complexes.
Let $G$ be a simple graph on $n$ vertices with $c$ connected components, and let $e(G)$ be the number of edges of $G$.
Then, the first Betti number $\beta_1(G)=e(G)-n+c$, which implies that
$$e(G) \le n+\beta_1(G)-c.$$
So, if $G$ has the maximum number of edges among all graphs over $\mathcal{K}(n,1,\beta)$, then $c=1$, namely, $G$ is connected.

By Lemma \ref{pathcon}, we can also deduce the following result on the maximum
spectral radius.

\begin{coro}\label{s-con}
Let $K$ be a pure $r$-dimensional complex on $n$ vertices with maximum spectral radius $\q_{r-1}(K)$ among all complexes over $\mathcal{K}(n,r,\beta)$.
Then $K$ contains all possible $(r-1)$-faces and is $(r-1)$-path connected.
\end{coro}

Corollary \ref{s-con} is also a generalization of the graph case.
If $G$ is a graph with maximum spectral radius $\q_0(G)$ among all graphs over $\mathcal{K}(n,1,\beta)$, then $G$ must be connected.
Otherwise, we can add some edges to $G$ such that the resulting graph $G'$ is connected and has the same first Betti number as $G$.
However, $\q_0(G') > \q_0(G)$ by the Perron-Frobenius theorem, yielding a contradiction.

In the following, we present an upper bound for the spectral radius of a pure $r$-dimensional complex without $\Delta_{r+2}^{r+1}$, and characterize the equality case by using the $r$-neighbor uniform property.
An $r$-complex $K$ is called \emph{$r$-neighbor uniform}, if for any $r$-face $F$ and any vertex $u \notin F$, $|N^\dn(F,u)|=r$, where
\begin{equation}\label{NdFu}
 N^\dn(F,u):=\left\{G\cup\left\{u\right\}\in S_r(K):G\in\partial F\right\},
\end{equation}
that is, $N^\dn(F,u)$ is the set of down-neighbors of $F$ that contain the vertex $u$.

\begin{thm}\label{main0}
Let $K$ be a pure $r$-dimensional complex on $n$ vertices without $\Delta_{r+2}^{r+1}$.
Then
$$\q_{r-1}(K)\le rn-r^2+1,$$
with equality if and only if  $K$ is $r$-neighbor uniform.
\end{thm}

We characterize the structure of a pure $r$-dimensional complex which is $r$-neighbor uniform.
An $r$-dimensional \emph{tented complex} on $n$ vertices, denoted by $\T_n^{r+1}$, is a pure $r$-dimensional complex whose facets are all $(r+1)$-subsets of $[n]$ that contain a fixed vertex; see Fig. \ref{TR} for the tented complex $\T_6^3$ for illustration.

\begin{thm}\label{NeiU}
Let $r\geq 2$ be an integer. Let $K$ be a pure $r$-dimensional simplicial
complex on $n$ vertices, which is $r$-neighbor uniform.
If $r$ is even, or $r$ is odd and $n \geq (r-2)!!(r+1)+r+2$, then $K \cong \wedge_n^{r+1}$.
\end{thm}

We finally arrive at the main result of this paper.

\begin{thm}\label{main}
Let $K$ be a pure $r$-dimensional complex on $n$ vertices without $\Delta_{r+2}^{r+1}$.
Then
\begin{equation}\label{Eq_beta0}
\q_{r-1}(K)\le rn-r^2+1.
\end{equation}
Furthermore,

{\em(1)} if $r=1$, the equality in \eqref{Eq_beta0} holds if and only if $K$ is a complete bipartite graph;

{\em(2)} if $r \ge 2$ is even, or $r$ is odd and $n \geq (r-2)!!(r+1)+r+2$, the equality in \eqref{Eq_beta0} holds if and only if  $K\cong \T_n^{r+1}$.
\end{thm}

By Remark \ref{oddRn}, if we have no requirement for $n$ when $r$ is odd, then $K$ contains a double pyramid $\Diamond_{r+3}^{r+1}$.
Surely, if a complex contains no homeomorphs of $S^r$, then it contains no $\Delta_{r+2}^{r+1}$ or $\Diamond_{r+3}^{r+1}$.
So, we have the following corollary from Theorem \ref{main}.

\begin{coro}\label{noSr}
Let $K$ be a pure $r$-dimension complex on $n$ vertices without homeomorphs of  $S^r$.
Then
\begin{equation}\label{sphere}
\q_{r-1}(K)\le rn-r^2+1,
\end{equation}
with equality holds if and only if $ K\cong\T_n^{r+1}$.
\end{coro}

If a complex contains no $r$-holes, then it contains no homeomorphs of $S^r$. Therefore, we also have the following corollary.

\begin{coro}\label{main-hole}
Let $K$ be a pure $r$-dimension simplicial complex on $n$ vertices without  $r$-holes.
Then
$$
\q_{r-1}(K)\le rn-r^2+1,
$$
with equality holds if and only if $K \cong \T_n^{r+1}$.
\end{coro}

It is known that for a graph $G$ on $n$ vertices without cycles (namely, forests), $\q_0(G) \le n$, with equality if and only if $G$ is a star, namely, $\T_n^2$.
So Corollary \ref{main} generalizes the graph case.

Note that $\T_n^{r+1}$ is a cone over the complete $(r-1)$-skeleton a simplicial complex on $n-1$ vertices.
Consequently, for all $i=1,\ldots,r$, $\beta_i(\T_n^{r+1})=0$, which implies that $\T_n^{r+1}$ contains no $i$-holes for any $i \in [r]$.
So, we can also deduce the following result.

\begin{coro}\label{main2}
Let $K$ be a pure $r$-dimension simplicial complex on $n$ vertices without any $i$-holes for $i \in [r]$.
Then
$$
\q_{r-1}(K)\le rn-r^2+1,
$$
with equality holds if and only if $K \cong \T_n^{r+1}$.
\end{coro}

Finally, we provide an upper bound on the signless Laplacian spectral radii of $r$-dimensional complexes with a prescribed Betti number.

\begin{thm}\label{gen}
Let $K$ be a pure $r$-dimensional complex on $n$ vertices with $\beta_{r}(K)=t$,
where $1 \le t\le n-r-1$.
Then
\begin{equation}\label{max=beat=t}
\q_{r-1}(K)\le rn-r^2+t+1.
\end{equation}
\end{thm}

By a relationship between the number of facets of a pure $r$-dimensional complex $K$ and the spectral radius $\q_{r-1}(K)$ of $K$ (see Lemma \ref{connec}),
we obtain an upper bound for the Tur\'an number $\ex(n, \D_{r+1}^r)$ in Corollary \ref{HypTur}, and exact numbers  $\ex(n, \D_{5}^4)$ for $n=6,7,8$ in Remark \ref{Dconterexample}.
We can also derive some upper bounds for the Tur\'an density in Remark \ref{Density} or for the Tur\'an number of $S^r$ in Remark \ref{TuranS}.

The paper is organized as follows.
In Section \ref{sec2}, we introduce some basic notions about hypergraphs, simplicial complex and homology group.
In Section \ref{sec3-1}, we prove a  basic property of an $r$-complex that maximizes the facet number or the spectral radius among all $r$ complexes in $\mathcal{K}(n,r,\beta)$ (Corollaries \ref{e-con} and \ref{s-con}).
In Section \ref{sec3-2}, we determine the $r$-complex that is $r$-neighbor uniform (Theorem \ref{NeiU}).
In Section \ref{sec3-3}, we give an upper bound for the spectral radius of an $r$-complex $K$ without $\Delta_{r+2}^{r+1}$, and prove that $K$ attains the upper bound if and only if it is $r$-neighbor uniform.
By using Theorem \ref{NeiU}, we show that the tented complex is the unique complex without $\Delta_{r+2}^{r+1}$ that maximizes the spectral radius.
Our technique involves the canonical Alexander dual of perfect matchings and coloring of simplicial complexes.

\section{Preliminaries}\label{sec2}

\subsection{Hypergraph, simplicial complex and homology group}

A \emph{hypergraph} $H$ is a pair $(V,E)$ consisting of a vertex set $V=:V(H)$ and an edge set $E=:E(H)$, where each element of $E$ is a subset of $V$.
If each edge $e \in E$ has cardinality $r$, then $H$ is called \emph{$r$-uniform}.
An \emph{abstract simplicial complex} (or simply a \emph{complex}) $K$ on a finite set $V$ is a collection of subsets of $V$ that is closed under inclusion.
An \emph{$i$-face} or an \emph{$i$-simplex} of $K$ is an element of $K$ with cardinality $i+1$.
The collection of all $i$-faces of $K$ is denoted by $S_i(K)$.
Usually, we call the $0$-faces the \emph{vertices} of $K$, and the $1$-faces the \emph{edges} of $K$.
The \emph{dimension of an $i$-face} is $i$, and the \emph{dimension of a complex $K$} is the maximum dimension of a face in $K$.
The faces that are maximal under inclusion are called \emph{facets}.
We say $K$ is \emph{pure} if all facets have the same dimension.
By definition, a complex $K$ can be viewed as a hypergraph with all facets as edges, and a hypergraph can also be viewed as a complex with all subsets of edges forming the collection.

For two $(i+1)$-simplices of a complex $K$ sharing an $i$-face, we refer to them as \emph{$i$-down neighbors}.
For two $i$-simplices of $K$ which are faces of an $(i+1)$-simplex, we call them \emph{$(i+1)$-up neighbors}.
The complex $K$ is said to be \emph{$i$-path connected} if, for any two $i$-faces $F$ and $G$ of $K$, there exists an ordered sequence of $i$-faces $F_1,F_2,\ldots,F_m$ such that $F_i$ and $F_j$ are $(i+1)$-up neighbors if and only if $|i-j|=1$, where $F=F_1$ and $G=F_m$.
The above sequence of $i$-faces is called an \emph{$i$-path} connecting $F$ and $G$.
For a face $F$ of $K$, we denote by $N^{\dn}(F)$ the set of down neighbors and by
$N^{\un}(F)$ the set of up neighbors of $F$, respectively.
The \emph{degree} of an $i$-face $F$ is the number of $(i+1)$-faces of $K$ that contain $F$, denoted by $\deg(F)$.

A face $F$ is \emph{oriented} if we choose an ordering on its vertices and write $[F]$.
Two orderings of the vertices are said to determine the same orientation if there is an even permutation transforming one ordering into the other.
If the permutation is odd, then the orientations are opposite.
The \emph{$i$-th chain group} $C_i(K,\R)$ of $K$ with coefficients in $\R$ is a vector space over the field $\R$ with basis $B_i(K,\R)=\{[F]\mid F\in S_i(K)\}$.
The \emph{$i$-th boundary map} $\partial_{i}:C_{i}(K,\mathbb{R})\to C_{i-1}(K,\mathbb{R})$ is defined by \[\partial_i[v_{0},\cdots,v_{i}]=\sum_{j=0}^{i}(-1)^{j}[v_{0},\cdots,\hat{v}_{j},\cdots ,v_{i}],\]
where $\hat{v}_{j}$ indicates that the vertex $v_j$ is omitted.
This gives rise to the chain complex of $K$:
\[\cdots \longrightarrow C_{i+1}(K,\mathbb{R})\overset{\partial_{i+1}}{\longrightarrow}
C_i(K,\mathbb{R})\overset{\partial_{i}}{\longrightarrow}C_{i-1}(K,\mathbb{R})
\longrightarrow \cdots \longrightarrow C_0(K,\mathbb{R})\longrightarrow 0,\]
satisfying $\partial_i\circ\partial_{i+1}=0$.
The kernel of $\p_i$, denoted by  $\ker \partial_{i}$, is called the \emph{group of $i$-cycles}.
The image of $\p_{i+1}$, denoted by $\im\partial_{i+1}$, is called the \emph{group of $i$-boundaries}.
The quotient
$H_i(K)=\ker\partial_{i}/\im \thinspace\partial_{i+1}$
 is called the \emph{$i$-th homology group of $K$}.
Its dimension, denoted by $\beta_i(K)$, is called the \emph{$i$-th Betti number} of $K$.
A nonzero element of $H_i(K)$ is referred to as an \emph{$i$-dimensional hole} or simply an \emph{$i$-hole}.

The \emph{$i$-th cochain group} $C^i(K,\R)$ is defined as the dual space of $C_i(K,\R)$, i.e.,  $C^i(K,\R)= \text{Hom}(C_i(K,\R),\R)$, which is generated by the  dual basis $B_i^*(K,\R)=\{[F]^*: F \in S_i(K)\}$, where $[F]^*([F])=1$ and $[F]^*([F'])=0$ for $F' \ne F$.
The \emph{$i$-th coboundary map} $\d_{i}:C^{i}(K,\mathbb{R})\to C^{i+1}(K,\mathbb{R})$ is defined by $\d_{i}f=f\p_{i+1}$, that is,
\[(\d_{i}f)[v_{0},\cdots,v_{i+1}]=\sum_{j=0}^{i+1}(-1)^{j}f([v_{0},\cdots,\hat{v}_{j},\cdots ,v_{i+1}]).\]
	
Endowing $C^i(K,\mathbb{R})$ and $C^{i+1}(K,\mathbb{R})$ with positive definite inner products, respectively, we obtain
the adjoint $\delta_i^* : C^{i+1}(K,\mathbb{R}) \to C^i(K,\mathbb{R})$ of $\delta_i$, which is defined by
\[
(\delta_i f_1, f_2)_{C^{i+1}} = (f_1, \delta_i^* f_2)_{C^i}
\]
for all $f_1 \in C^i(K,\mathbb{R})$, $f_2 \in C^{i+1}(K,\mathbb{R})$.
	
Three Laplace operators \cite{HJ2013} on $C_i(K,\R)$ are defined as
\begin{equation}\label{Laplace}
L_i^{\up}(K)=\d_{i}^*\d_{i}, ~ L_i^{\down}(K)=\d_{i-1}\d_{i-1}^*, ~
L_i(K)=\d_{i}^*\d_{i} + \d_{i-1}\d_{i-1}^*
\end{equation}
which are referred to as  the \emph{$i$-th up Laplace operator}, \emph{$i$-th down Laplace operator} , and \emph{$i$-th Laplace operator} of $K$, respectively \cite{DR2002}.
Eckmann \cite{ECK1944} proved the discrete version of the Hodge theorem (see also \cite{DR2002,HJ2013}), which can be formulated as
$\ker L_i(K) \cong H_i(K,\R)$.

Finally, we note that extensive work has been done on the spectrum of the Laplacian of complexes and its relation to combinatorial properties; see
\cite{DR2002,HJ2013,GOL2017,BGP2019,FSW2023,FWW2025,SWF2025}

\subsection{Signless Laplacian of a simplicial complex}
The signless Laplacian of a simplicial complex was introduced in \cite{KO2020}, and it was systematically studied with many interesting applications \cite{KL2014,Lub2014}.
It was further extended to the signless $1$-Laplacian for investigating the combinatorial properties of a complex \cite{LZ2020}.
In addition, one can refer to \cite{SWF2025} for normalized Laplacians eigenvalues.

Consider the vector space $D_i(K,\R)$ over $\R$ generated by all $i$-faces of $K$.
The \emph{$i$-th signless boundary map}
 $|\p_i|: D_i(K,\R) \to D_{i-1}(K,\R)$ is defined by
$$ |\p_i|\{v_0,\ldots,v_i\}=\sum_{j=0}^i \{v_0,\ldots,\hat{v}_j, \ldots, v_i\},$$
where $\hat{v}_j$ denotes the omission of the vertex $v_j$.
Let $D^i(K,\R)$ be the dual space of $D_i(K,\R)$ that is generated by dual basis $\{F^*: F \in S_i(K)\}$.
The \emph{$i$-th signless coboundary map} $|\d_i|: D^i(K,\R) \to D^{i+1}(K,\R)$ is defined by $ |\d_i| f =f |\p_{i+1}|$ for each $f \in D^i(K,\R)$.

For each $i$, endow $D^i(K,\R)$ with a positive inner product that makes
 the dual basis $\{F^*: F \in S_i(K)\}$ orthonormal. (One can define a general positive inner product over $D^i(K,\R)$.)
Let $|\d_i|^*:D^{i+1}(K,\R) \to D^i(K,\R)$ be the adjoint of $|\d_i|$,
which satisfies
$$\langle |\d_i| f, g \rangle_{D^{i+1}(K,\R)}=\langle f, |\d_i|^* g \rangle_{D^{i}(K,\R)}$$
for all $f \in D^i(K,\R)$ and $g \in D^{i+1}(K,\R)$.
The \emph{$i$-th up signless Laplace operator} and \emph{$i$-th down signless Laplace operator} of $K$ are respectively defined by
\begin{equation}\label{SLaplace}
Q_i^{\up}(K)=|\d_{i}|^* |\d_{i}|, ~ Q_i^{\down}(K)=|\d_{i-1}| |\d_{i-1}|^*.
\end{equation}

Since $Q_i^{\up}(K)$ shares the same eigenvalues as $Q_{i+1}^{\down}(K)$, except for the zero eigenvalue, we focus on discussing $Q_i^{\up}(K)$ only.
By definition, for $f \in D^i(K,\R)$,
\begin{equation}\label{ABd}
 |\d_{i}| f = \sum_{\bar{F}\in S_{i+1}(K)} \left(\sum_{F \in \p \bar{F}} f(F) \right) \bar{F},
\end{equation}
where $\p \bar{F}$ denotes the set of $i$-faces in the boundary of $\bar{F}$.
Denote
$$f(\p \bar{F}):=\sum_{F: F \in \p \bar{F}} f(F).$$
Thus, by \eqref{ABd}, we have
\begin{equation}\label{Qev}
(Q_i^{\up}(K)f)(F)=\sum_{\bar{F}\in S_{i+1}(K): F \in \p \bar{F}} f(\p \bar{F})=
\deg(F) f(F) + \sum_{F' \in N^\un(F)} f(F').
\end{equation}
Similarly, we have
\begin{equation}\label{Qevdown}
(Q_i^{\down}(K)f)(F)=|F| f(F) + \sum_{F' \in N^\dn(F)} f(F').
\end{equation}
If $f$ is an eigenvector of $Q_i^{\up}(K)$ associated with an eigenvalue $\la$,
then, by \eqref{Qev}, for any $\bar{F}\in S_{i+1}(K)$,
\begin{equation}\label{Q-equ-F}
\begin{split} \la f(\p \bar{F}) &=\sum_{F \in \p \bar{F}} \la f(F) \\
&=\sum_{F \in \p \bar{F}} (Q_i^{\up}(K)f)(F) \\
&=\sum_{F \in \p \bar{F}} \sum_{\bar{F'}: F \in \p \bar{F'}}f(\p \bar{F'})\\
&=|\bar{F}|f(\p \bar{F})+\sum_{\bar{F'} \in N^\dn(\bar{F})}f(\p \bar{F'}).
\end{split}
\end{equation}

In addition, by \eqref{ABd}, we have
\begin{equation}\label{Qform}
\langle Q_i^{\up}(K)f, f \rangle=\langle |\d_{i}|^* |\d_{i}| f, f \rangle
=\langle |\d_{i}| f , |\d_{i}| f \rangle
= \sum_{\bar{F}\in S_{i+1}(K)} f(\p \bar{F})^2,
\end{equation}
Observe that $Q_i^{\up}(K)$ is nonnegative, symmetric, and positive semidefinite, and it is irreducible if and only if $K$ is $i$-path connected.
By the Perron-Frobenius theorem for nonnegative matrices,
the largest eigenvalue of $Q_i^{\up}(K)$ is exactly its spectral radius, denoted by $\q_i(K)$.
So, by \eqref{Qform}, we have the following expression for the spectral radius:
\begin{equation}\label{Q-equ-r}
\q_i(K)=\max_{f \in D_i(K,\R)\setminus \{0\}} \frac{\langle Q^{\up}_i(K)f, f \rangle}{\langle f, f \rangle}=
\max_{f \in D_i(K,\R)\setminus \{0\}} \frac{\sum_{\bar{F}\in S_{i+1}(K)} f(\p\bar{F})^2}{\langle f, f \rangle}.
\end{equation}
If $Q_i^{\up}(K)$ is further irreducible, then, up to a positive scalar,  there exists a unique positive eigenvector associated with $\q_i(K)$, called the \emph{Perron vector} of $Q_i^{\up}(K)$.

By using Eq. \eqref{Q-equ-r} and taking $f$ to be an all-one vector, we obtain the following result, which provides a connection between the spectral radius and the number of facets.

\begin{lemma}\label{connec}
Let $K$ be a pure $(i+1)$-dimensional complex that contains all possible $i$-faces. Then
$$|S_{i+1}(K)| \le \frac{\q_{i}(K)\binom{n}{i+1}}{(i+2)^2}.$$
\end{lemma}

\section{Spectral radius of  complexes without holes}\label{sec3}
In this section, we will investigate the signless Laplacian spectral radius $\q_{r-1}(K)$ of a pure $r$-dimensional complex $K$.
For simplicity, we call $\q_{r-1}(K)$ the \emph{spectral radius of $K$}.

\subsection{Connectedness}\label{sec3-1}
In this section, we will show that if $K$ is a pure $r$-dimensional complex whose spectral radius achieves the maximum among all complexes in $\mathcal{K}(n,r,\beta)$, then $K$ contains all possible $(r-1)$-faces and is $(r-1)$-path connected.

Before proceeding with the proof, we first introduce some preliminary definitions.
The \emph{ $i$-th incidence graph} of  a complex $K$, denoted by $B_i(K)$, is a bipartite graph with vertex set $S_i(K) \cup S_{i+1}(K)$, where $\{F,\bar{F}\}$ is an edge of $B_i(K)$ if and only if $F \in \p \bar{F}$.
It follows that $K$ is $i$-path connected if and only if $B_i(K)$ is connected.

We give a geometric interpretation of $r$-holes.
Let $h$ be an $r$-hole of an $r$-dimensional complex $K$, viewed as an element of $H_r(K)$.
Then $h$ is a linear expression of oriented $r$-faces of $K$.
Let $\supp(h)$ denote the set of $r$-faces that appear with nonzero coefficients in the linear expression of $h$, and let $K_h$ denote the subcomplex of $K$ induced by $\supp(h)$.
Clearly, $\beta_r(K_h) \ge 1$.
We call $K_h$ a \emph{basic $r$-hole} if $\beta_r(K_h)=1$.
Equivalently, a basic $r$-hole $H$ is a minimal pure $r$-dimensional complex with $\beta_r(H)=1$, in the sense that removing any $r$-face $F$ from $H$ yields $\beta_r(H-F)=0$.

For $r=1$ (graph case), a basic $1$-hole is exactly a cycle.
However, there are many possibilities for two-dimensional or higher holes.
For example, the triangulation of a $2$-dimensional sphere and that of a $2$-dimensional torus are both basic $2$-holes.
We have the following general result from the intuition of cycles, spheres, and tori.

\begin{lemma}\label{hole}
Let $K$ be a basic $r$-hole, where $r \ge 2$.
Then the following results hold.

{\em (1)} $K$ is $(r-1)$-path connected.

{\em (2)} $d_K(F)\ge 2$ for any $(r-1)$-face $F$ of $K$.

{\em (3)} $K - \bar{F} $ is $(r-1)$-path connected for any $r$-face $\bar{F}$ of $K$.

\end{lemma}

\begin{proof}
(1) Let $\p_r: C_r(K,\R) \to C_{r-1}(K,\R)$ be the boundary map over $K$.
By definition, the kernel of $\p_r$ is given by $\ker \p_{r}=\{t \sigma:t\in\mathbb{R}\}$,
where $\sigma=\sum_{\bar{F}\in S_r(K)}\sigma_{\bar{F}}[\bar{F}]$.
By the minimal property of a basic hole,
\begin{equation}\label{coef}
\sigma_{\bar{F}}\ne 0, \mbox{~for~all~} \bar{F}\in S_r(K).
\end{equation}
If $K$ is not $(r-1)$-path connected, then $B_{r-1}(K)$ is not connected.
Therefore, there exist two subcomplexes $K_1$ and $K_2$ of $K$ such that $C_r(K,\R)=C_r(K_1,\R) \oplus C_r(K_2,\R)$ and $K_1,K_2$ share no common $(r-1)$-faces.
Write $\sigma=\sigma_1 + \sigma_2$, where $\sigma_i \in C_r(K_i,\R)$ for $i=1,2$.
Note that $\sigma_i \ne 0$ for $i=1,2$ by \eqref{coef}.
From the equality
$ \p \s =\p \s_1 + \p \s_2 =0$
and the fact that $K_1$ and $K_2$ share no common $(r-1)$-faces, we obtain
$$ \p \s_1 = \p \s_2 =0,$$
which contradicts the minimal property of a basic hole.

(2) Let $\s$ be defined as in (1).
If there exists an $F\in S_{r-1}(K)$ such that $d_K(F)=1$, i.e., there is a unique $\bar{F} \in S_r(K)$ such that $F \in \p \bar{F}$, then
\begin{equation}\label{mini} (\p_r \sigma) ([F]) = \sigma_{\bar{F}} \sgn ([F], \p [\bar{F}]) \ne 0,
\end{equation}
which leads to a contradiction, where $\sgn ([F], \p [\bar{F}])$ denotes the sign of $[F]$ in $\p [\bar{F}]$.

(3) Assume for the contrary that $K':=K-\bar{F}$ is not  $(r-1)$-path connected for some $r$-face $\bar{F}$ of $K$.
Then there exist two subcomplexes $K'_1, K'_2$ of $K'$ such that
$C_r(K',\R)=C_r(K'_1,\R) \oplus C_r(K'_2,\R)$
 and $K'_1,K'_2$ share no common $(r-1)$-faces.
As $K$ is $(r-1)$-path connected,
\begin{equation}\label{2-con}
\p \bar{F} = \p \bar{F} |_{K'_1} \cup \p \bar{F} |_{K'_2},
\end{equation}
where $\p \bar{F} |_{K'_i}:=\p \bar{F} \cap S_{r-1}(K'_i)$ for $i=1,2$,
and each term of the union is nonempty.

Let $\bar{K}_i=K'_i + \bar{F}$ for $i=1,2$, which is obtained from $K'_i$ by adding $\bar{F}$ as a new $r$-face.
By the definition of a basic hole, we have $\beta_r(K'_i)=0$ for $i=1,2$.
We assert that $\beta_r(\bar{K}_1)=0$.
Otherwise,
there would be a basic $r$-hole $O$ in $\bar{K}_1$, which necessarily
contains the face $\bar{F}$ as $\beta_r(K'_1)=0$.
By \eqref{2-con}, any $(r-1)$-face $F'$ in $\p \bar{F}|_{K'_2}$ has degree $1$ in $O$, which contradicts the result in (2).
A similar argument shows that $\beta_r(\bar{K}_2)=0$.
Thus, we have
$$ K= \bar{K}_1 \cup \bar{K}_2, \bar{F}= \bar{K}_1 \cap \bar{K}_2.$$
It is known that $\beta_r(\bar{F})=\beta_{r-1}(\bar{F})=0$ when $r \ge 2$.
So, by the Mayer-Vietoris sequence (see \cite[\S 25]{Munk})
$$ H_r(K) \cong H_r(\bar{K}_1) \oplus H_r(\bar{K}_2).$$
As $\beta_r(\bar{K}_i)=0$ for $i=1,2$,
we conclude that $\beta_r(K)=0$, which is a contradiction.
The result follows.
\end{proof}

\begin{proof}[\bf Proof of Lemma \ref{pathcon}]
If $K$ itself contains all possible $(r-1)$-faces and is $(r-1)$-path connected, we take $L=K$.
Otherwise,
let $\bar{K}$ be the complex obtained from $K$ by adding all possible $(r-1)$-faces.
In this case, $\bar{K}$ is not $(r-1)$-path connected.
Consider the two incidence graphs $B_{r-1}(\bar{K})$ and $B_{r-2}(\bar{K})$.
As $\bar{K}$ contains all possible $(r-1)$-faces, $B_{r-2}(\bar{K})$ is connected.
Suppose $\bar{K}_1$ is an $(r-1)$-path connected component of $\bar{K}$.
Since $B_{r-2}(\bar{K})$ is connected, there must exist an $(r-1)$-face $F_1$ of $\bar{K}_1$, and an $(r-1)$-face $F_2$ of another $(r-1)$-path connected component of $\bar{K}$, say $\bar{K}_2$, such that $F_1 \cap F_2$ is an $(r-2)$-faces of $\bar{K}$.
In other words, there exists an $(r-2)$-face $E$ such that $\{E, F_1\}$ and $\{E, F_2\}$ are both edges of $B_{r-2}(\bar{K})$, or equivalently, $E \in \p F_1$ and $E \in \p F_2$.
Now adding a new $r$-face $\bar{F}:=F_1 \cup F_2$ to $\bar{K}$, we get
a new complex $\bar{K}+ \bar{F}$ such that $\bar{K}_1$ is connected to $\bar{K}_2$ in this new complex.

We assert that $\bar{K}+ \bar{F}$ contains no new $r$-holes except those of $K$.
Otherwise, there exists a basic $r$-hole $O$ which necessarily contains the face $\bar{F}=F_1 \cup F_2$.
By Lemma \ref{hole}, $O$ is $(r-1)$-path connected.
Hence, $O$ must be contained in an $(r-1)$-path connected component of $\bar{K}+\bar{F}$, which contains both $\bar{K}_1$ and $\bar{K}_2$.
Furthermore, by Lemma \ref{hole}, $O-\bar{F}$ is also $(r-1)$-path connected.
In particular, $F_1$ is connected to $F_2$ by an $(r-1)$-path in $O-\bar{F}$.
However,  $F_1 \in \bar{K}_1$ and $F_2 \in \bar{K}_2$ are disconnected  in $\bar{K}=(\bar{K}+\bar{F})-\bar{F}$, leading to a contradiction.

We can repeat this process, adding new $r$-faces until the resulting complex is $(r-1)$-path connected,
The final complex, $L$, will have the desired properties as stated in the lemma.
\end{proof}

\begin{proof}[\bf Proof of Corollary \ref{s-con}]
Assume to contrary, $K$ does not contain all possible $(r-1)$-faces or $K$ is not $(r-1)$-path connected.
By Lemma \ref{pathcon}, there exists a complex $L$ containing $K$ as a proper subcomplex, such that $L$ contains all possible $(r-1)$-faces and is $(r-1)$-path connected.
Furthermore, we have $\beta_r(K)=\beta_r(L)$.
Surely, $Q_{r-1}^{\up}(L)$ is nonnegative and irreducible.
Note that $Q_{r-1}^{\up}(K)$ is either a proper principal submatrix of $Q_{r-1}^{\up}(L)$ (if $K$ does not contain all possible $(r-1)$-faces) or  $Q_{r-1}^{\up}(K) \lneqq Q_{r-1}^{\up}(L)$ (if $K$ contains all possible $(r-1)$-faces but is not $(r-1)$-path connected).
By Perron-Frobenius Theorem, this implies that $\q_{r-1}(K) < \q_{r-1}(L)$, which is  a contradiction.
\end{proof}

\subsection{$r$-neighbor uniformity}\label{sec3-2}
In this section We give some structural properties of $r$-neighbor uniform $r$-complexes.

\begin{lemma}\label{conne}
If $K$ is $r$-neighbor uniform $r$-complex, then it is $r$-path connected.
\end{lemma}

\begin{proof}
Let $F_0$ be an $r$-face of $K$.
Let $K'$ be an $r$-path connected component of $K$ that contains the face $F_0$.
Assume to the contrary that $K$ is not $r$-path connected.
Then there exists an $r$-face $G \notin K'$.
Let $F$ be an $r$-face of $K'$ that maximizes $|F \cap G|$.
As $G \notin K'$, $G$ is not an down neighbor of $F$, implying that $|F \cap G| \le r-1$.
Pick a vertex $u \in G \setminus F$.
By the $r$-uniform condition, $N^d(F,u)|=r$.
So we have a vertex $v \in F\setminus G$ such that $F_v:=(F \setminus \{v\}) \cup \{u\} \in N^\dn(F,u)$.
However,
$$ |F_v \cap G| = |F \cap G|+1,$$
contradicting the maximality of $|F \cap G|$.
So, $K$ is $r$-path connected.
\end{proof}

We give an example of $r$-neighbor uniform $r$-complex, namely the canonical Alexander dual of a perfect matching on $r+3$ vertices.
Let $V$ be a set with $(r+3)$ vertices, where $r$ is odd.
Let $M$ be a perfect matching of $V$, namely, a partition of $V$ with each part having $2$ elements.
The \emph{canonical Alexander dual} $M^\vee$ of $M$ (simply called the \emph{dual of $M$}) is defined by
$$ M^\vee := \{V \setminus e: e \in \binom{V}{2} \setminus M\}.$$
By definition, $M^\vee$ is a pure $r$-dimensional complex whose $r$-faces are the complements of $2$-subsets of $V$ except those of $M$,
or equivalently,
$$ S_r(M^\vee)=\binom{V}{r+1} \setminus \{V \setminus e: e \in M\}.$$

Denote by $\Delta_V^2$ the complete graph on the vertex set $V$.
Then $M$ is a perfect matching of $\Delta_V^2$, namely, a set of pairwise disjoint edges that covering all vertices.
By definition, $F \in S_r(M^\vee)$ if and only if $F = V \setminus e$ for some edge $e \notin M$.
Write $F_e:= V \setminus e$ for an edge $e \notin M$.
Then, two faces $F_e$ and $F_{e'}$ are down neighbors if and only if $|e \cap e'|=1$, namely, $e$ intersects $e'$ into exactly one vertex.

An $r$-dimensional \emph{double pyramid}, denoted by $\Diamond_{r+3}^{r+1}$, is an $r$-dimensional complex on $r+3$ vertices with the facet set:
\[\left\{\{i\} \cup F:  i\in \{r+2,r+3\}, F\in \binom{[r+1]}{r}\right\}.\]
Note that $\Diamond_{r+3}^{r+1}$ is a suspension of the boundary of an $r$-simplex; see Fig. \ref{TR} for the double pyramid $\Diamond_5^3$ for illustration.
The polyhedron of the boundary of an $r$-simplex is homeomorphic to $S^{r-1}$,  and consequently, the polyhedron of $\Diamond_{r+3}^{r+1}$ is homeomorphic to $S^r$.

\begin{figure}[htbp]
\begin{center}
\includegraphics[scale=1.2]{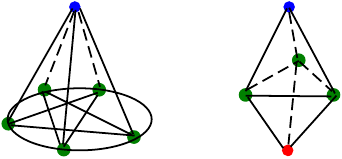}
\end{center}
\caption{\small Tented complex $\T_6^3$ (left) and rhombic complex $\Diamond_5^3$ (right)}\label{TR}
\end{figure}

\begin{lemma}\label{dualM}
Let $V$ be a vertex set of size $r+3$, where $r$ is odd.
Let $M$ be a perfect matching of $V$, and let $M^\vee$ be the canonical Alexander dual of $M$.
Then $M^\vee$ is $r$-neighbor uniform, and contains $\frac{r+3}{2}$ copies of double pyramids $\Diamond_{r+3}^{r+1}$.
\end{lemma}	

\begin{proof}
Let $F_e:=V \setminus e$ be an $r$-face of $M^\vee$.
Then $e = V \setminus F \notin M$.
Suppose that $e =\{x,y\}$.
Let $F_{e'} \in N^{\dn}(F,x)$.
Then $e'$ intersects $e$ into exactly one vertex, $e'$ contains no the vertex $x$.
So, $e'$ contains the vertex $y$, which implies that
 $$N^{\dn}(F,x) = \{F_{e'}: e' \notin M, y \in e', x \notin e'\}.$$
The number of edges $e'$ in $\Delta_V^2$ incident with $y$, except the edge of $M$ incident with $y$ and the edge $e$, is exactly $r+2-2=r$.
So, $|N^{\dn}(F,x)| = r$.
Similarly $|N^{\dn}(F,y)| = r$.
Thus $M^\vee$ is $r$-neighbor uniform.

For each edge $e=\{x,y\} \in M$, let $R_e=V \setminus e$.
By definition, $R_e$ is not a facet of $M^\vee$.
For any $r$-subset of $G \subset R_e$, letting $z \in R_e \setminus G$,   thus $F_{\{y,z\}}=\{x\} \cup G$ and $F_{\{x,z\}}=\{y\} \cup G$ are both facets of $M^\vee$.
So, we have a double pyramid $\Diamond_{r+3}^{r+1}$ with two apexes $x,y$.
As $M$ contains $\frac{r+3}{2}$ edges, $M^\vee$ contains $\frac{r+3}{2}$ copies of double pyramids $\Diamond_{r+3}^{r+1}$.
\end{proof}

\begin{lemma}\label{TM}
Let $K$ be a pure $r$-dimension complex with vertex set $V$ of size $r+3$, which is $r$-neighbor uniform.
Then $K\cong \wedge_{r+3}^{r+1}$, or $K \cong M^\vee$ for some perfect matching $M$ of $V$ only in the case of odd $r$.
\end{lemma}	

\begin{proof}
Suppose that the vertex set  $V=\{1,\ldots,r+2,r+3\}$.
Note that for any $r$-face $F \in K$, $V \setminus F$ is a $2$-subset of $V$.
Denote $F_{\{i,j\}}:=V\setminus \{i, j\}$ for $i,j \in V$ with $i \ne j$.
So, any face of $K$ can be represented as $F_e$ for some $2$-subset $e$ of $V$.
Let $G$ be a graph on $V$ whose edges are those $e$ such that $F_e \in S_r(K)$.
Let $F_e$ be an $r$-face of $K$, where $e=\{u,v\} \in E(G)$.
Recall that
$$N^{\dn}(F_e,v)=\{F_{e'} \in S_r(K): e' \in E(G), |e' \cap e|=1, v \notin e'\}.$$
So,  $|N^{\dn}(F_e,v)|$ is the number of edges $e'$ of $G$ incident with $u$ except the edge $e$.
By the condition, $|N^{\dn}(F_e,v)|=r$, so $u$ has degree $r+1$ in $G$, including the contribution of the edge $e$.
Similarly, by considering $N^{\dn}(F_e,u)$, $v$ has degree $r+1$ in $G$.
So, each endpoint of an edge of $G$ has degree $r+1$.

Without loss of generality, let $F_0=F_{\{r+2,r+3\}}$ be a face of $K$.
Then $\{r+2,r+3\}$ is an edge of $G$.
Let $G'$ be a connected component of $G$ that contains the vertices $r+1$ and $r+2$.
By the above discussion, the vertices $r+1$ and $r+2$ both have degree $r+1$ in $G$.
So, $G'$ has at least $r+2$ vertices.

{\bf Case 1.} $G' \ne G$.
As $G$ has $r+3$ vertices, $G'$ has exactly $r+2$ vertices, and $G$ is a disjoint union of $G'$ and an isolated vertex, say the vertex $1$.
By the previous discussion, $G'$ is $(r+1)$-regular.
So $G'$ is a complete graph with vertex set $[r+1]$.
This implies that
$$S_r(K)=\left\{F_e: e \in \binom{[r+1]}{2}\right\},$$
and hence each $r$-face of $K$ contains the vertex $1$.
So, $K \cong \wedge_{r+3}^r$.

{\bf Case 2.} $G'=G$.
In this case, $G$ is a connected $(r+1)$-regular graph.
As $G$ has $r+3$ vertices, each vertex of $G$ is incident with a missed edge.
Furthermore, any two missed edges are disjoint; otherwise, the common endpoint of the two missed edges would have degree less than $r+1$.
So, those missed edges form a perfect matching $M$ of $V$, implying that $r$ is odd
and $G=\Delta_V^2-M$.
So,  $K \cong M^\vee$.
\end{proof}

Next we prove that the an $r$-neighbor uniform $r$-complex $K$ can only be a cone complex if it has a large number of vertices.
We begin the discussion for $r=3$.

\begin{lemma}\label{Case3}
Let $K$ be a pure $3$-dimensional simplicial complex on $n \ge 9$ vertices, which is $3$-neighbor uniform.
 Then $ K\cong \wedge^3_n$.
\end{lemma}

\begin{proof}
Let $V$ be the vertex set of $K$.
The condition that $K$ is $3$-neighbor uniform is equivalent to the following local  condition.

\begin{claim}[$\{0,4\}$-condition] For every $5$-subset $X\subseteq V$, the number $3$-faces of $K[X]$ is either $0$ or $4$.
\end{claim}

Indeed, if $X$ contains a $3$-face $F$, then $X=F\cup\{u\}$ for the
unique vertex $u\in X\setminus F$.
In $K[X]$, there are $3$ facets in $N^\dn(F,u)$ containing the vertex $u$, and $1$ facet  without $u$ (namely the face $F$).
Thus $K[X]$ contains exactly $1+3=4$ facets.
If $X$ contains no facet, the number of facets in $K[X]$ is $0$.

For a facet $F$ of $K$, we define a coloring $c$ on $V\setminus F$ with respect to $F$ as follows.
For each vertex $x\in V\setminus F$, the $5$-set $F \cup\{x\}$ contains
exactly $4$ facets. Since $F$ itself is a facet, among the four
replacement faces
$(F\setminus\{u\})\cup\{x\},~ u \in F$,
exactly one is not a facet.
Hence, we may define a color $c(x)\in F$ by requiring that
\[
(F\setminus\{c(x)\})\cup\{x\}\notin K,
\]
while the other three replacement faces belong to $K$.

\begin{claim}\label{3ver}
For a facet $F$ of $K$ and any three distinct vertices $x,y,z\in V\setminus F$,
 the colors of $x,y,z$ are either all equal or pairwise distinct.
\end{claim}

Let $F=:\{a,b,c,d\}$.
Suppose to the contrary, without loss of generality, that
\[
c(x)=c(y)=a,~ c(z)=b, ~ a\neq b.
\]
By Lemma \ref{TM}, $K[F\cup \{x,y\}]$ is a cone complex with the apex $a$.
So, $\{a,b,x,y\} \in K$.
Also by Lemma \ref{TM}, $K[F\cup \{x,z\}]$ is a dual of a matching $M$ of $F\cup \{x,z\}$,
where $M=\{\{a,z\},\{b,x\},\{c,d\}\}$.
So, $\{a,b,x,z\}\notin K$.
Similarly, by considering the complex $K[F\cup \{y,z\}]$, we have
$\{a,b,y,z\}\notin K$.
Now consider the $5$-set
$$ X=\{a,b,x,y,z\}.$$
Thus $K[X]$ contains at least $1$ facet and at least $2$ missed facets, contradicting the $\{0,4\}$-condition. Hence the claim holds.
Since $V\setminus F$ at least $5$ vertices but there are only $4$ possible
colors from $F$, two of them must have the same color.
By Claim \ref{3ver}, all vertices in $V\setminus F$ must have the same color.

Without loss of generality, let $F_0=\{1,2,3,4\}$ be a facet of $K$, and
assume that this common color with respect to $F_0$ is $1$.
We now prove $K$ is a cone with apex $1$.
As the common color with respect to $F_0$ is $1$, every down neighbor of $F_0$
contains the vertex $1$.
Let $F$ be a down-neighbor of $F_0$. Then $F$ contains $1$.
We claim that the common color relative to $F$ is also $1$.
Write $F=(F_0\setminus\{j\})\cup\{x\}$ for some $j\neq 1$ and $x \notin F_0$.
Take arbitrarily a vertex $z \in V\setminus(F_0\cup\{x\})$,
Since the common color relative to $F_0$ is $1$, the induced complex on $F_0\cup\{x,z\}$ is a cone with apex $1$ by Lemma \ref{TM}.
Hence the $4$-set $(F\setminus\{1\})\cup\{z\}$ is not a facet, because it does not contain $1$.
So, the color of $z$ relative to $F$ is $1$, and hence the common color relative to $F$ is $1$.

As $K$ is $r$-path connected by Lemma \ref{conne}, each facet $F$ of $K$ can be connected to $F_0$ by an $r$-path of facets.
So, repeating the above argument along the $r$-path, we have the common color relative to  $F$ is $1$.
We conclude that every facet of $K$ contains the vertex $1$.
Hence $K$ is a cone with apex $1$.

Finally, we prove $K$ is a complete cone over $1$, namely, any $3$-subset $T$ of $V \setminus \{1\}$ is a base of the cone.
We will use the operation called \emph{single-vertex replacement}.
Let $F=\{1\}\cup B$ be a facet, and let $v \in F$.
For any $u\in V\setminus F$, replacing the vertex $v$ of $F$ by $u$ will give a new $(r+1)$-set:
$$ F_v:=(F \setminus \{v\}) \cup \{u\}.$$
Surely, $F_1$ is not a facet, since every facet of $K$ contains $1$.
Therefore the $\{0,4\}$-condition forces all other replacements
$F_b$ for all $b \in B$ to be facets.
Thus the collection of bases $B$ of $K$ is closed under all single-vertex replacements.
Since any $3$-subset $T$ of $V\setminus\{1\}$ can be obtained from $B$ by a sequence of single-vertex replacement, $T$ is a base.
Consequently, $K\cong \wedge_n^3$.
\end{proof}

We now prove a general result, namely Theorem \ref{NeiU}, for $r$-neighbor uniform $r$-complexes by a similar discussion as Lemma \ref{Case3}.

\begin{proof}[\bf Proof of Theorem \ref{NeiU}]
The result for $r=3$ is proved in Lemma \ref{Case3}.
We assume that $r \ge 5$ in the following discussion.
Let $V=[n]$ be the vertex set of $K$.
The condition that $K$ is $r$-neighbor uniform is equivalent to the following local \emph{$\{0, r+1\}$-condition}:
every $(r+2)$-subset $X\subseteq V$ contains either $0$ or $r+1$ facets of $K$.
Indeed, if $X$ contains a facet $F$, then $X=F\cup\{u\}$ for the unique
vertex $u\in X\setminus F$.
There are $r$ facets in $N^\dn(F,u)$ containing the vertex $u$, and $1$ facet in $X$ without $u$ (i.e. the face $F$).
 Hence $X$ contains exactly $r+1$ facets.

Fix a facet $F_0=\{1,2,\dots,r+1\}\in S_r(K)$.
For every outside vertex $x\in V\setminus F_0$, define its color
$c(x)\in F_0$ by
\[
(F_0\setminus\{c(x)\})\cup\{x\}\notin S_r(K).
\]
This is well-defined because $F_0\cup\{x\}$ contains exactly one missed facet.
Since
$|V\setminus F_0|\geq (r-2)!!(r+1)+1$,
the pigeonhole principle gives a color class of size at least $(r-2)!!+1$.
Relabeling if necessary, choose distinct vertices
$x_1,\dots,x_{(r-2)!!+1}\in V\setminus F_0$ such that
$c(x_i)=1$ for all $i \in [(r-2)!!+1]$.

We first show that all vertices in $V\setminus F_0$ have the same color.
Suppose to the contrary that there exists vertices outside $F_0$ with different colors.  Choose another outside vertex $y$ with $c(y)=2$.
For each $i\in [(r-2)!!+1]$, set $V_i=F_0\cup\{y,x_i\}$.
By Lemma \ref{TM}, the induced complex $K[V_i]$ is a dual of a perfect matching $M_i$ on $V_i$, where
\[
M_i=\{\{1,y\},\{2,x_i\}\}\cup M_i',
\]
and $M_i'$ is a perfect matching on $F_0\setminus\{1,2\}$.

Here, we note this case only occurs when $r$ is odd; otherwise, $F_0\setminus\{1,2\}$ would not contain perfect matchings, a contradiction.
There are exactly $(r-2)!!$ perfect matchings on $F_0\setminus\{1,2\}$, while we have
$ (r-2)!!+1 $ vertices $x_i$.
Hence, by the pigeonhole principle, two of these induced
matchings have the same restriction on $F_0\setminus\{1,2\}$.
Relabeling, assume $ M_1' = M_2'=:M'$.
As $r\geq 5$, $M'$ contains at least two distinct edges, say $A=\{a,a'\}$ and $B=\{b,b'\}$.

We now derive a contradiction from the local $\{0,r+1\}$-condition.
First consider an $(r+2)$-set:
$$W_{aa'}=(F_0\setminus \{a,a'\})\cup\{y,x_1,x_2\}.$$
Since $A=\{a,a'\}\in M'$, both $V_1\setminus A$ and $V_2\setminus A$
are missed facets of $K$, which are both contained in $W_{aa'}$.
Hence $W_{aa'}$ contains at least two missed facets.
By the local $\{0,r+1\}$-condition, $W_{aa'}$ cannot contain any facet at all.
Therefore every $(r+1)$-subset of $W_{aa'}$ is a missed facet.

Next consider another $(r+2)$-set:
\[
W_{a'b'}=(F_0 \setminus \{a',b'\})\cup \{y,x_1,x_2\}.
\]
Noting that $W_{a'b'}\setminus\{a\}=W_{aa'} \setminus \{b'\}$, which is a $(r+1)$-subset of $W_{aa'}$,  it is a missed facet of $W_{a'b'}$ by the previous discussion.
On the other hand, for $i=1,2$, the set
$(F_0 \setminus \{a',b'\})\cup\{y,x_i\}$ is a facet of $K[V_i]$, because its complementary edge $\{a',b'\}$ in $V_i$ is not an edge of the matching $M_i$.
Thus $W_{aa'}$ contains at least two facets.
By the local $\{0,r+1\}$-condition, $W_{a'b'}$ contains exactly one non-facet, namely
$W_{a'b'}\setminus\{a\}$.
Hence $W_{a'b'} \setminus\{b\}$ is a facet of $K$.
Similarly, by considering the set:
$$ W_{ab}=(F_0 \setminus \{a,b\})\cup \{y,x_1,x_2\},$$
we have $W_{ab} \setminus\{b'\}$ is a facet of $K$.

Finally consider the $(r+2)$-set:
\[
W_{bb'}=(F_0 \setminus \{b,b'\})\cup \{y,x_1,x_2\}.
\]
Since $B=\{b,b'\}\in M'$, the two $(r+1)$-subsets $V_1\setminus B$ and $V_2\setminus B$
are both non-facets of $W_{bb'}$.
Observe that
$W_{bb'} \setminus \{a'\}=W_{a'b'} \setminus \{b\}$ and
$ W_{bb'} \setminus \{a\}=W_{ab} \setminus\{b'\}$, which are both facets of $K$ by the previous discussion.
Therefore, $W_{bb'}$ contains at least two facets and at least
two missed facets; contradicting the local $\{0, r+1\}$ condition.
Thus the coloring on $V\setminus F_0$ must be constant.
Relabeling, we may assume the common color is $1$.

We now prove that $K$ is a complete cone with apex $1$. The above argument applies
to every facet, not only to $F_0$. Hence, for every facet $F$, all outside
vertices have one common color relative to $F$.
By a similar discussion as in the proof of Lemma \ref{Case3}, the common color relative to every facet $F$ is $1$, implying that $K$ is a cone over $1$.
By using single-vertex replacements as in proof of Lemma \ref{Case3},
any $r$-subset $B$ of $V\setminus\{1\}$ is a base of the cone $K$.
Therefore, $K\cong \wedge_n^{r+1}$.
\end{proof}

\begin{rmk}\label{oddRn} \rm
In Lemma \ref{Case3} and Theorem \ref{NeiU}, if we have no requirement for $n$ when $r$ is odd, from their proof we know that either $K \cong \T_n^{r+1}$ or $K$ contains a subcomplex $M^\vee$ for some perfect matching $M$ on $r+3$ vertices, where the former corresponds the case that the coloring relative every facet is constant, and the latter corresponds the case the coloring relative some facet is not constant.
By Lemma \ref{dualM}, $M^\vee$, and hence $K$, contains a double pyramid $\Diamond_{r+3}^{r+1}$ if $K \not\cong \T_n^{r+1}$.
\end{rmk}

\subsection{Spectral radius}\label{sec3-3}
In this section, we first introduce two special complexes: tented complexes and double pyramid complexes.
We will give an upper bound for the signless Laplacian spectral radius of  a pure $r$-dimensional complex $K$ without $\D_{r+2}^{r+1}$,
and show that if a complex $K$ achieves the upper bound, then $K$ is either a tented complex or contains a rhombic complex.
By this result, we get upper bounds for the Tur\'an numbers of $\D_{r+1}^r$, and for the spectral radius of a complex without a homeomorph of $S^r$ or a $r$-hole.
Finally, we provide an upper bound for the spectral radius of a complex with a prescribed Betti number.

\begin{lemma}\label{rNeiSpe}
Let $K$ be a pure $r$-complex on $n$ vertices, which is $r$-neighbor uniform.
Then
\[\q_{r-1}(K)=rn-r^2+1.\]
\end{lemma}

\begin{proof}
It suffices to prove $Q_r^{\down}(K)$ has the spectral radius $rn-r^2+1$, as $Q_{r-1}^{\up}$ shares the same nonzero eigenvalues with $Q_r^{\down}(K)$.
Let $\mathbf{1}$ be an all-one vector defined on $S_r(K)$.
Since $K$ is $r$-neighbor uniform, for each facet $F \in S_r(K)$ and each vertex $u \in V(K) \setminus F$, $|N^{\dn}(F,u)|=r$, which implies that $F$ has $r(n-r-1)$ down neighbors.
So, by \eqref{Qevdown},
$$ (Q_r^{\down}(K) \mathbf{1})(F) =|F|\cdot 1+ \sum_{F'\in N^{\dn}(F)} 1 =r+1+ r(n-r-1)=rn-r^2+1.$$
Thus, $Q_r^{\down}(K) \mathbf{1}=(rn-r^2+1) \mathbf{1}$.
By Perron-Frobenius theorem for nonnegative matrices, $Q_r^{\down}(K)$, and hence $Q_{r-1}^{\up}$ has the spectral radius $rn-r^2+1$.
\end{proof}

As the tented complex $\T_n^{r+1}$ and the double pyramid $\Diamond_{r+3}^{r+1}$ are both $r$-neighbor uniform, by Lemma \ref{rNeiSpe}, we immediately get the following result.

\begin{coro}\label{tent}
\[\q_{r-1}(\T_n^{r+1})=rn-r^2+1, ~ \q_{r-1}(\Diamond_{r+3}^{r+1})=3r+1.\]
\end{coro}

\begin{proof}[\bf Proof of Theorem \ref{main0}]
Suppose that $K_0$ is a pure $r$-dimensional complex on $n$ vertices, which does not contain $\D_{r+2}^{r+1}$, and that its spectral radius $\q_{r-1}(K_0)$ attains the maximum.
By Corollary \ref{s-con}, $K_0$ contains all possible $(r-1)$-faces and is $(r-1)$-path connected.

Let $f$ be the Perron vector of $Q_{r-1}^{\up}(K_0)$ associated with the spectral radius $\q_{r-1}(K_0)$.
Recall that for a face $F \in S_r(K_0)$, we define $f(\p F):=\sum_{G \in \p F} f(G)$.
By normalization, let $F_0 \in S_r(K_0)$ be such that
\[f(\p F_0)=\max\left \{ f(\p F):F\in S_r(K_0)\right \}=1.\]
Since $K_0$ does not contain $\D_{r+2}^{r+1}$, for any $u \in V(K) \setminus F_0$, we have
\begin{equation}\label{localu}
|N^{\dn}(F_0,u)|\le r,
\end{equation}
 and hence
 \begin{equation}\label{local}
 |N^\dn(F_0)| =\sum_{u \in V(K) \setminus F_0} |N^{\dn}(F_0,u)|\le r(n-r-1).
 \end{equation}
By \eqref{Q-equ-F} and the above inequality, we get
\begin{equation}\label{ev2}
\begin{aligned}
q_{r-1}(K_0)f(\p F_0)&=|F_0|f(\p F_0)+\sum_{F \in N^d(F_0)}f(\p F)\\
&\le r+1+|N^{d}(F_0)|\\
&\le r+1+r(n-r-1)=rn-r^2+1.
\end{aligned}
\end{equation}

Clearly, $\T_n^{r+1}$ contains no $\D_{r+2}^{r+1}$, and by Corollary \ref{tent},
  $$\q_{r-1}(K_0)\ge q_{r-1}(\T_n^{r+1})=rn-r^2+1.$$
Thus, all inequalities in \eqref{ev2} hold as equalities,
which implies that
for each $u \in V(K) \setminus F_0$, $|N^\dn(F_0,u)|=r$, and for each $F \in N^\dn(F_0)$, $f(\p F)=1$.
Repeating this argument for each $F \in N^\dn(F_0)$ and its down neighbors, and using the fact that $K_0$ is $(r-1)$-path connected, we obtain
$|N^\dn(F,u)|=r$ for each $F \in S_r(K_0)$ and each $u \in V(K_0)\setminus F$, namely, $K_0$ is $r$-neighbor uniform.

Conversely, if $K$ is $r$-neighbor uniform, then, by Lemma \ref{rNeiSpe}, $q_{r-1}(K)=rn-r^2+1$.
The sufficiency follows.
\end{proof}

\begin{proof}[\bf Proof of Theorem \ref{main}]
By Theorem \ref{main0}, it suffices to characterize the complex $K$ satisfying the equality in \eqref{Eq_beta0}.
If $r=1$, $K$ is a simple graph, and $\D_{3}^2$ is a triangle.
Taking an edge $e=\{u,v\}$ of $K$, by $1$-neighbor uniformity,
each vertex of $K$ other than $u,v$ is connected to $u$ or $v$, but not both.
Therefore, $N(u)$ and $N(v)$ form a bipartition of $V(K_0)$, where $N(w)$ denotes the set of neighbors of a vertex $w$ in $K$.
As $K$ is triangle-free, there are no edges within $N(u)$ or $N(v)$, which implies that $K_0$ is a bipartite graph.
By the maximality of the signless Laplacian radius, $K$ must be a complete bipartite graph.

Now suppose $r \ge 2$ in the following.
By Theorem \ref{main0},  $K$ is $r$-neighbor uniform, and hence by Theorem \ref{NeiU}, we get the desired result immediately.
\end{proof}

\begin{rmk}\label{Dconterexample}
\rm
We give a remark for Theorem \ref{NeiU} and Theorem \ref{main}.
If $r$ is odd and $n < (r-2)!!(r+1)+r+2$, the complex $\T_n^{r+1}$ may not be the unique $r$-complex to be $r$-neighbor uniform or to satisfy the equality in Eq. \eqref{Eq_beta0}.
Here, we present all examples for $r=3$ and $ 4 \le n \le 8$.

For $n=4$, there is only one pure $3$-dimensional complexes of on $4$ vertices, namely the $3$-simplex together with its boundary, which trivially satisfies the equality in Eq. \eqref{Eq_beta0}.

For $n=5$, the complex $\T_5^{4}$ is the unique one achieving the equality in Eq. \eqref{Eq_beta0}.

For $n=6$, by Lemma \ref{TM}, there are exactly $2$ non-isomorphic pure $3$-complexes on $6$ vertices that attain equality in Eq. \eqref{Eq_beta0}:  $\T_6^4$, and the dual of a perfect matching $M$. For example, letting $M=\{\{1,6\},\{2,5\},\{3,4\}\}$, then
$$S_3(M^\vee)=\binom{[6]}{4}\setminus \left\{1256,
2345,1346\right\}.$$
Here and in the below, we simply write a $4$-set $\{a,b,c,d\}$ as $abcd$.

It is straightforward to verify that $M^\vee$ contains $2$ independent $3$-holes both isomorphic to $\Diamond_{r+3}^{r+1}$ and $\beta_3(K_1)=2$.
While $\T_6^4$ and $M^\vee$ share the same spectral radius $10$, they differ in the number of facets:  $|S_3(\T_6^4)|=10$ and $|S_3(M^\vee)|=12$.
By Lemma \ref{connec} and Theorem \ref{main0} (or Corollary \ref{HypTur} below),
$$\ex(6,\D_{5}^4) \le \frac{25}{2},$$
which implies that $\ex(6,\D_{5}^4) =12$ because $|S_3(M^\vee)|=12$.

For $n=7$, there are exactly $2$ non-isomorphic pure $3$-complexes on $7$ vertices that attain equality in Eq. \eqref{Eq_beta0}: $\T_7^4$, and the complex $K_1$ with facet set:
$$\binom{[7]}{4}\setminus \left\{1256,
2345,1346, 1247, 1357, 2367, 4567\right\}.$$

The complex $K_1$ contains $8$ independent $3$-holes all isomorphic to $\Diamond_{r+3}^{r+1}$ and $\beta_3(K_2)=8$.
Also, $\T_7^4$ and $K_1$ have same spectral radius $13$ but differ in number of facets:  $S_3(\T_7^4)=20$ and $S_3(K_1)=28$.
Similarly, $$ \ex(7,\D_{5}^4) \le \frac{455}{16},$$
which implies that $\ex(7,\D_{5}^4) =28$ because $|S_3(K_1)|=28$.

For $n=8$, except $\T_8^4$, up to isomorphism, the following complex $K_2$ is the unique one satisfying \eqref{Eq_beta0} whose facets are $\displaystyle\binom{[8]}{4}\setminus \mathcal{B}$,
where $\mathcal{B}$ is the block set of a Steiner quadruple system $S(3,4,8)$ of size $14$, listed as follows:
 \[
\mathcal B=\{
1234,1256,1278,1357,1368,1458,1467,
2358,2367,2457,2468,3456,3478,5678
\}.
\]

The complex $K_2$ contains $21$ independent $3$-holes all isomorphic to $\Diamond_{r+3}^{r+1}$ and $\beta_3(K_2)=21$.
Also, $\T_8^4$ and $K_2$ have same spectral radius $16$ but differ in number of facets:  $S_3(\T_8^4)=35$ and $S_3(K_2)=56$.
Similarly, $$ \ex(8,\D_{5}^4) \le 56,$$
which implies that $\ex(7,\D_{5}^4) =56$ because $|S_3(K_2)|=56$.
\end{rmk}

By Lemma \ref{connec} and Theorem \ref{main0}, we obtain an upper bound for the hypergraph Tur\'an number $\ex(n,\D_{r+1}^r)$.

\begin{coro}\label{HypTur}
$$ \textup{ex}(n,\D_{r+1}^r) \le \frac{(r-1)(n-r+1)+1}{r(n-r+1)}\binom{n}{r}.$$
\end{coro}

\begin{rmk}\label{Density}
\em
By Corollary \ref{HypTur}, we have
$$ \ex(n,\D_{4}^3) \le \frac{1}{9}n^3- o(n^3),$$
which shows a subtle discrepancy from the bound $\frac{5}{54}n^3 +o(n^3)$ derived from Tur\'an's conjecture (Eq. \eqref{Turan}).
When considering Tur\'an's density, we have
\[\pi(\D_{r+1}^{r})\le \frac{r-1}{r},\]
which was proven by De Caen \cite{DE1983}.
\end{rmk}

\begin{rmk}\label{TuranS}
\em
By Lemma \ref{connec} and Corollary \ref{noSr}, we obtain an upper bound for the Tur\'an number of the homeomorph of $S^r$ as follows:
\begin{equation}\label{sbound}
 \ex(n,S^r) \le \Theta(n^{r+1}).
 \end{equation}
 The upper bound in \eqref{sbound} is not as sharp as those given by
 Brown, Erd\H{o}s, and S\'os in Eq. \eqref{nS2} for $r=2$,
  and by Newman and Pavelka \cite{NP2024} in Eq. \eqref{nSr} for a general $r$.
Therefore, there is a discrepancy between the spectral bound and the structural bound.
From Corollary \ref{main2} below, the complex $\T_n^{r+1}$ attaining the upper bound in \eqref{sphere} contains no any $i$-holes for $i \in [r]$, which is a very strong condition that holds in the equality case.
\end{rmk}

Finally, we prove the upper bound on the signless Laplacian spectral radii of $r$–dimensional complexes with a prescribed Betti number.

\begin{proof}[\bf Proof of Theorem \ref{gen}]
By Lemma \ref{pathcon}, we can assume that $K$ is $(r-1)$-path connected.
Let $f$ be the Perron vector of $Q_{r-1}(K)$ corresponding to the eigenvalue $q_{r-1}(K)$.
Assume that an $r$-face $F_0$ of $K$ satisfies
\[f(\p F_0)=\max\left\{f(\p F):F\in S_{r}(K)\right\}=1.\]
By \eqref{Q-equ-F}, we have
\begin{equation}\label{tbeta}
\begin{aligned}
 \q_{r-1}(K) &=\q_{r-1}(K)f(\p F_0)\\
 &=|F_0| f(\p F_0) + \sum_{F \in N^\dn(F_0)}f(\p F)\\
 &\le r+1+|N^{d}(F_0)|\\
 &\le r+1+(n-r-1)r+t=rn-r^2+t+1,
 \end{aligned}
 \end{equation}
where the last inequality follows from the fact that $K$ contains at most $t$ complexes of type $\D_{r+2}^{r+1}$, each containing the face $F_0$, given by the fact $\beta_r(K)=t$.
\end{proof}


\begin{thebibliography}{99}

\bibitem{BGP2019}
C. Bachoc, A. Gundert, A. Passuello,
\newblock The theta number of simplicial complexes,
\newblock \emph{Israel J. Math.}, 232: 443-481, 2019.



\bibitem{CL1999}
F. Chung, L. Lu,
\newblock An upper bound for the Tur\'an number $t_3(n,4)$,
\newblock \emph{J. Combin. Theory Ser. A}, 87(2): 381-389, 1999.

\bibitem{DE1983}
D. de Caen,
\newblock Extension of a theorem of Moon and Moser on complete subgraphs,
\newblock \emph{Ars Combin.}, 16: 5-10, 1983.


\bibitem{DR2002}
A. M. Duval, V. Reiner,
\newblock Shifted simplicial complexes are Laplacian integral,
\newblock \emph{Trans. Amer. Math. Soc.}, 354(11): 4313-4344, 2002.

\bibitem{ECK1944}
B. Eckmann,
\newblock Harmonische funktionen und randwertaufgaben in einem komplex,
\newblock \emph{Comment. Math. Helv.}, 17(1): 240-255, 1944.


\bibitem{EFR1986}
P. Erd\H{o}s, P. Frankl, V. R\H{o}dl.
\newblock The asymptotic number of graphs not containing a fixed subgraph and a
  problem for hypergraphs having no exponent,
\newblock \emph{Graphs Combin.}, 2: 113-121, 1986.

\bibitem{ERD1966} P. Erd\H{o}s, M. Simonovits,
\newblock A limit theorem in graph theory,
\newblock \emph{Studia Sci. Math. Hungar.}, 1: 51–57, 1966.

\bibitem{EGM2019}
B. Ergemlidze, E. Gy\H{o}ri, A. Methuku,
\newblock Asymptotics for Tur\'an numbers of cycles in $3$-uniform linear hypergraphs,
\newblock \emph{J. Combin. Theory Ser. A}, 163: 163-181, 2019.


\bibitem{FSW2023}
Y.-Z. Fan, Y.-M. Song, Y. Wang,
\newblock  The spectra of Laplace operators on covering simplicial complexes,
\newblock \emph{arXiv: 2312.12709}, 2023.

\bibitem{FWW2025}
Y.-Z. Fan, H.-F. W, Y. Wang,
\newblock The largest Laplacian eigenvalue and the balancedness of simplicial complexes,
\newblock \emph{J. Algebraic Combin.}, 61: 53, 2025.

\bibitem{FG2020}
Z. F\"uredi, A. Gy\'arf\'as,
\newblock An extension of Mantel’s theorem to $k$-graphs,
\newblock \emph{Amer. Math. Monthly}, 127(3): 263-268, 2020.

\bibitem{FO2017}
Z. F\"uredi, L. \"Ozkahya,
\newblock On $3$-uniform hypergraphs without a cycle of a given length,
\newblock \emph{Discrete Appl. Math.}, 216: 582-588, 2017.


\bibitem{GC2021}
G. Gao, A. Chang,
\newblock A linear hypergraph extension of the bipartite Tur\'an problem,
\newblock \emph{European J. Combin.}, 93: 103269, 2021.

\bibitem{GCH2022}
G. Gao, A. Chang, Y. Hou.
\newblock Spectral radius on linear r-graphs without expanded $K_{r+1}$,
\newblock \emph{SIAM J. Discrete Math.}, 36(2): 1000-1011, 2022.

\bibitem{GMV2019}
D. Gerbner, A. Methuku, and M. Vizer,
\newblock Asymptotics for the Tur\'an number of Berge-$K_{2,t}$,
\newblock \emph{J. Combin. Theory Ser. B}, 137: 264-290, 2019.

\bibitem{GP2017}
D. Gerbner, C. Palmer,
\newblock Extremal results for Berge hypergraphs,
\newblock \emph{SIAM J. Discrete Math.}, 31(4): 2314-2327, 2017.

\bibitem{GOL2017}
K. Golubev,
\newblock On the chromatic number of a simplicial complex,
\newblock \emph{Combinatorica}, 37: 953-964, 2017.

\bibitem{HJ2013}
D. Horak, J. Jost,
\newblock Spectra of combinatorial Laplace operators on simplicial complexes,
\newblock \emph{Adv. Math.}, 244: 303-336, 2013.

\bibitem{HCC2021}
Y. Hou, A. Chang, J. Cooper,
\newblock Spectral extremal results for hypergraphs,
\newblock \emph{Electron. J. Combin.}, 28(3): P3.46, 2021.

\bibitem{KNS1964}
G. Katona, T. Nemetz and M. Simonovits,
\newblock On a problem of Tur\'an in the theory of graphs,
\newblock \emph{Mat. Lapok}, 15(1): 228-238, 1964.

\bibitem{KO2020}
T. Kaufman, I. Oppenheim,
\newblock High order random walks: Beyond spectral gap,
\newblock \emph{Combinatorica}, 40: 245-281, 2020.

\bibitem{KL2014}
 T. Kaufman, A. Lubotzky,
\newblock High dimensional expanders and property testing,
\newblock \emph{Proc. 5th Conf. Innov. Theor. Comput. Sci.}, pp. 501-506, 2014.

\bibitem{KLM2014}
P. Keevash, J. Lenz, D. Mubayi,
\newblock Spectral extremal problems for hypergraphs,
\newblock \emph{SIAM J Discrete Math.}, 28(4): 1838-1854, 2014.

\bibitem{KLNS2021}
P. Keevash, J. Long, B. Narayanan, A. Scott,
\newblock A universal exponent for homeomorphs,
\newblock \emph{Israel J. Math.},  243: 141–154, 2021.

\bibitem{KPTZ2022}
A. Kupavskii, A. Polyanskii, I. Tomon, D. Zakharov,
\newblock The extremal number of surfaces,
\newblock \emph{Int. Math. Res. Not.}, 17: 13246--13271, 2022.


\bibitem{LV2003}
F. Lazebnik, J. Verstra\"ete,
\newblock On hypergraphs of girth five,
\newblock \emph{Electron. J. Combin.}, 10: \#R25, 2003.

\bibitem{LNY2022}
J. Long, B. Narayanan, C. Yap,
\newblock Simplicial homeomorphs and trace-bounded hypergraphs,
\newblock \emph{Discrete Anal.}, 6, 2022.


\bibitem{LZ2009}
L. Lu, Y. Zhao,
\newblock An exact result for hypergraphs and upper bounds for the Tur\'an density of $K_{r+1}^r$,
\newblock \emph{SIAM J Discrete Math}, 23(3): 1324-1334, 2009.

\bibitem{Lub2014}
A. Lubotzky,
\newblock Ramanujan complexes and high dimensional expanders,
\newblock \emph{Japan. J. Math.}, 9: 137-169, 2014.


\bibitem{LZ2020}
X. Luo, D. Zhang,
\newblock Spectrum of signless $1$-Laplacian on simplicial complexes,
\newblock \emph{Electron. J. Combin.}, 27(2): \# P2.30, 2020.


\bibitem{M2006}
D. Mubayi, O. Pikhurko,
\newblock A hypergraph extension of Tur\'an's theorem,
\newblock \emph{J. Combin. Theory Ser. B}, 96(1): 122-134, 2006.

\bibitem{MP2007}
D. Mubayi, O. Pikhurko,
\newblock A new generalization of Mantel's theorem to $k$-graphs,
\newblock \emph{J. Combin. Theory Ser. B}, 97(4): 669-678, 2007.

\bibitem{Munk}
J. R. Munkres,
\newblock  \emph{Elements of Algebraic Topology},
\newblock  CRC Press, Boca Raton, 1984.


\bibitem{NP2024}
A. Newman, M. Pavelka.
\newblock A conditional lower bound for the Tur\'an number of spheres.
\newblock \emph{\em arXiv: 2403.05364}, 2024.

\bibitem{RS1978}
I. Z. Ruzsa, E. Szemer\'edi.
\newblock Triple systems with no six points carrying three triangles,
\newblock In \emph{Combinatorics (Proc. Fifth Hungarian Colloq.,
  Keszthely, 1976), Vol. II}, pages 939-945, Amsterdam-New York, 1978.

\bibitem{San}
Maya Sankar,
\newblock The Tur\'an number of surfaces,
\newblock \emph{Bull. London Math. Soc.}, 56: 3786–3800, 2024.

\bibitem{SFKH2023}
C.-M. She, Y.-Z. Fan, L.~Kang, Y.~Hou,
\newblock Linear spectral Tur\'an problems for expansions of graphs with
given chromatic number,
\newblock \emph{Acta Math. Appl. Sin. Engl. Ser.}, 2025, doi: 10.1007/s10255-024-1156-x.

\bibitem{SFK2025}
C.-M. She, Y.-Z. Fan, L. Kang,
\newblock Spectral bipartite Tur\'an problems on linear hypergraphs,
\newblock \emph{Discrete Math.}, 348(6): 114435, 2025.

\bibitem{SWF2025}
Y.-M. Song, H.-F. Wu, Y.-Z. Fan,
\newblock The normalized Laplacian eigenvalue and incidence balancedness of simplicial complexes,
\newblock \emph{Bull. Iran. Math. Soc.}, 51: 42, 2025.


\bibitem{SEB1973}
T. S\'os,  P. Erd\H{o}s,  W. G. Brown,
\newblock On the existence of triangulated spheres in $3$-graphs, and related problems,
\newblock \emph{Period. Math. Hungar.}, 3: 221-228, 1973.

\bibitem{Tim2017}
C. Timmons,
\newblock  On $r$-uniform linear hypergraphs with no Berge-$K_{2,t}$,
\newblock  \emph{Electron. J. Combin.}, 24: \#P4.34, 2017.

\bibitem{Tur1941}
P. Tur\'an,
\newblock On an extremal problem in graph theory,
\newblock \emph{Mat. Lapok}, 48: 436-452, 1941.

\end{thebibliography}
\end{document}